\documentclass[11pt,twoside, a4paper, english, reqno]{amsart}
 \usepackage[foot]{amsaddr}
\usepackage[dvips]{epsfig}
\usepackage{amscd}
\usepackage{amssymb}
\usepackage{amsthm}
\usepackage{amsmath}
\usepackage{mathtools}
\usepackage{latexsym}

\usepackage{upref}
\usepackage{bm}
\usepackage{color}
\usepackage{hyperref}
\usepackage{graphicx}
\usepackage{lipsum}
\usepackage{comment}

\hypersetup{linkcolor=blue, colorlinks=true
,citecolor = red}

\theoremstyle{plain}
\newtheorem{Th}{Theorem}[section]
\newtheorem{Lem}[Th]{Lemma}
\newtheorem{Cor}[Th]{Corollary}

\theoremstyle{definition}

\newtheorem{Rem}[Th]{Remark}

\def\@setemails{%
  \ifnum\theg@author > 1
\mbox{{\itshape E-mail addresses}:\space}{\ttfamily\emails}. \else\mbox{{\itshape E-mail address}:\space}{\ttfamily\emails}. \fi%
}

\newcommand{\be} {\begin{equation}}
\newcommand{\ee} {\end{equation}}

\newcommand{\la} {\lambda}
\newcommand{\R} {\mathbb{R}}

\newcommand{\s} {\mathbb{S}}

\newcommand{\ds}{{\rm d}\sigma}

\newcommand{\authorfootnotes}{\renewcommand\thefootnote{\@fnsymbol\c@footnote}}%

\def\e{{\text{e}}}

\numberwithin{equation}{section} \allowdisplaybreaks

\title[Stability for Onofri-Type Inequality]{Quantitative stability  for the conformally invariant Chang-Gui inequality on the exponentiation of functions on the sphere}

\author[M. Ghosh]{Monideep Ghosh}

\author[D. Karmakar]{Debabrata Karmakar}

\email{monideep@tifrbng.res.in, debabrata@tifrbng.res.in}

 \address[]{Tata Institute of Fundamental Research, Centre For Applicable Mathematics}
 \address[]{Post Bag No 6503, GKVK Post Office,
Sharada Nagar, Chikkabommasandra,
Bangalore 560065,
Karnataka, India}



\begin{document}
\begin{abstract}
In this work, we focus on a recent variant of the Trudinger-Moser-Onofri inequality introduced by S. Y. Alice Chang and Changfeng Gui \cite{CG-2023}:
\begin{align*}
\alpha\int_{\s^2}|\nabla_{\s^2}u|^2 {\rm d}\omega+2 \int_{\s^2} u {\rm d}\omega -\frac{1}{2}\ln\left[\left(\int_{\s^2}e^{2u}{\rm d}\omega\right)^2-\sum_{i=1}^3\left(\int_{\s^2}\omega_i e^{2u}{\rm d} \omega\right)^2\right] \geq 0
\end{align*}
holds on $H^1(\s^2)$ if and only if $\alpha \geq \frac{2}{3}$. In this regime, the infimum is attained only by trivial functions when
$\alpha > \frac{2}{3},$ whereas for the critical value $\alpha = \frac{2}{3}$ nontrivial extremals exist, and Chang-Gui further provided a complete classification of such solutions.

Building upon their result, we found a nice conformal invariance of the associated functional. Exploiting this invariance, we were able to characterize the full family of extremals in terms of conformal maps of $\s^2$ and, moreover, establish a sharp quantitative stability result in the gradient norm.
 
\end{abstract}
\maketitle
{\small\tableofcontents}
\section{Introduction}
In this article, we investigate the stability properties of a variant of the Onofri inequality on $\mathbb{S}^2$, introduced by Chang and Gui \cite{CG-2023}. For $u \in H^1(\mathbb{S}^2)$, define
\begin{align}\label{Chang-GuiIneq}
I_{\alpha}(u) = \alpha\int_{\s^2}|\nabla_{\s^2}u|^2 {\rm d}\omega+2 \int_{\s^2} u {\rm d}\omega -\frac{1}{2}\ln\left[\left(\int_{\s^2}e^{2u}{\rm d}\omega\right)^2-\sum_{i=1}^3\left(\int_{\s^2}\omega_i e^{2u}{\rm d} \omega\right)^2\right].
\end{align}
Here and throughout, $\s^2$ denotes the standard two-dimensional sphere equipped with the round metric, $\nabla_{\s^2}$ is the gradient operator, ${\rm d}\omega$ is the normalized surface area measure, and $H^1(\s^2)$ is the Sobolev space of square-integrable functions with square-integrable gradient.

\medskip

Chang and Gui proved that $I_{\alpha}(u) \geq 0$ for all $u\in H^1(\s^2)$ if and only if $\alpha \geq \frac{2}{3}.$ More precisely, they established the sharp inequality
\begin{align}\label{CGineq}
I_{\alpha}(u) \geq \left(\alpha - \frac{2}{3}\right)\int_{\s^2}|\nabla_{\s^2}u|^2 {\rm d}\omega, 
\end{align}
valid for every $u \in H^1(\s^2)$. In addition, they showed that
 \begin{align*}
\inf_{u\in H^1(\s^2)} I_{\alpha}(u) = -\infty \ \ \  \mbox{if} \ \ \alpha < \frac{2}{3}.
 \end{align*}

 The inequality \eqref{CGineq} can be viewed as a generalization of the Lebedev–Milin inequality, extending the exponentiation of functions from the unit circle to the unit $2$-sphere. It may also be interpreted as the spherical analogue of the second inequality in Szeg\"{o}’s limit theorem for Toeplitz determinants on the circle.

 Let $D$ denote the unit disk in $\mathbb{R}^2 \sim \mathbb{C}$ centered at the origin, with boundary $\partial D = \mathbb{S}^1$. The classical Lebedev–Milin inequality can be written as
 \begin{align}
 \frac{1}{4\pi} \|\nabla u \|_{L^2(D)}^2 + \frac{1}{2\pi}\int_{\s^1} u {\rm d}\theta - \ln \left(\frac{1}{2\pi}\int_{\s^1} e^u {\rm d}\theta\right) \geq 0
 \end{align}
 for a harmonic function $u.$ 
 It is the first in a sequence of monotonically increasing inequalities in the Szeg\"{o} limit theorem \cite[5.5a]{GS-1958} on Toeplitz determinants. The special case of the second inequality in this sequence reads
\begin{align}
\frac{1}{8\pi} \|\nabla u \|_{L^2(D)}^2 + \frac{1}{2\pi}\int_{\s^1} u {\rm d}\theta - \ln \left(\frac{1}{2\pi}\int_{\s^1} \e^u {\rm d}\theta\right) \geq 0,
\end{align}
and holds for all functions $u$ satisfying $\int_{\s^1} e^u e^{i\theta} {\rm d}\theta=0$, where $i$ denotes the imaginary unit.

This inequality was independently verified by Osgood, Phillips, and Sarnak \cite{OPS-1988} in their study of isospectral compactness for metrics on compact surfaces. Later, H. Widom \cite{W-1988} observed that it is a direct consequence of the Szeg\"{o} limit theorem. More generally, there is a whole sequence of such inequalities for functions $u$ satisfying $\int_{\s^1} e^u e^{ij\theta} {\rm d}\theta=0$ for all $1 \leq j \leq k$ for some fixed $k.$

This perspective has been systematically explored on $\mathbb{S}^2$ in a series of recent works \cite{CH-2022,Hang-2022,CZ-2023} by Chang and Hang and others. 
Altogether, their approach provides a new method for proving a sequence of Lebedev-Milin-type inequalities on the unit circle, as well as a related sharp inequality via a perturbation argument. The method is flexible and can be readily adapted to functions subject to various boundary conditions or belonging to higher-order Sobolev spaces.


On the other hand, the Chang-Gui inequality can also be viewed as a variant of the classical Trudinger-Moser-Onofri and Trudinger-Moser-Aubin type inequalities on the sphere.


Consider the functional 
\begin{align} \label{TMAubin}
 J_{\alpha}(u)=\alpha \int_{\s^2} |\nabla u|^2 {\rm d} \omega + 2\int_{\s^2} u {\rm d} \omega -\ln{\left(\int_{\s^2}\e^{2u} {\rm d} \omega\right)},
\end{align}
defined on $H^1(\mathbb{S}^2)$. In his study of prescribing Gaussian curvature on $\mathbb{S}^2$ via variational methods, Moser \cite{M-1970, Moser-1973} proved the stronger inequality
\begin{align*}
\sup_{\|\nabla_{\s^2}u\|_2 = 1} \int_{\s^2} e^{4\pi u^2} {\rm d} \omega < \infty,
\end{align*}
from which it follows that $J_1(u)$ is bounded from below on $H^1(\mathbb{S}^2)$. The constant $4\pi$ is sharp and cannot be improved in general. However, if the functions are assumed to be antipodally symmetric, the constant can be improved to $8\pi$.

Later, Onofri \cite{O-1982}, in his investigation of the determinant of the conformal Laplacian, showed that in fact
 \begin{align}\label{Onofri}
 J_{1}(u)= \int_{\s^2} |\nabla u|^2 {\rm d} \omega + 2\int_{\s^2} u {\rm d} \omega -\ln{\left(\int_{\s^2}\e^{2u} {\rm d} \omega\right)} \geq 0,
 \end{align} 
 for all $u \in H^1(\mathbb{S}^2)$. Interestingly, \eqref{Onofri} arose in the study by Onofri and Virasoro \cite{OV-1982} in the context of Polyakov's \cite{P-1981} theory of random surfaces.
 
  Onofri’s proof exploits the conformal invariance of $J_1$ and a result of Aubin mentioned below: for a conformal map $\tau : \s^2 \to \s^2,$ and $u \in H^1(\s^2),$ define 
 \begin{align*}
 u_{\tau} = u \circ \tau + \frac{1}{2}\ln \mathcal{J}_{\tau},
 \end{align*}
where $\mathcal{J}_{\tau}$ is the Jacobian of $\tau.$ Then $J_1$ is invariant under the transformation $u \mapsto u_{\tau}$, and $J_1(u) = 0$ if and only if, up to an additive constant, $u = \frac{1}{2} \ln \mathcal{J}_{\tau}$. We refer the reader to Beckner’s article \cite{B-1993} for the Trudinger–Moser–Onofri inequality in higher dimensions, which is conformally invariant and relies on Lieb’s \cite{Lieb83} sharp form of the Hardy–Littlewood–Sobolev inequality on the sphere. While there have been further analogous developments in higher dimensions in the same spirit, for the sake of brevity we restrict our attention here to the two-dimensional case.
 
Historically, prior to Onofri’s result,  Aubin \cite{A-1979} observed that the critical value $\alpha = 1$ could be lowered under an additional vanishing moment condition: $J_\alpha(u)$ is bounded below for all $u \in H^1(\mathbb{S}^2)$ satisfying 
\begin{align*}
\int_{\s^2} we^{2u} {\rm d} \omega= 0,
\end{align*}
 if and only if $\alpha \geq \frac{1}{2}.$ In their influential work on the Gaussian curvature prescription problem, Chang and Yang \cite{CY-1987} observed that, for values of $\alpha$ close to $1$, the best lower bound of $J_\alpha$ remains $0$.

This result was later established in its sharp form for all admissible $\alpha$ by Gui and Moradifam \cite{GM-2018}, nearly three decades after the partial results and conjectures of Chang and Yang \cite[Prop. B]{CY-1987}, who were motivated by the goal of minimizing assumptions on the curvature function in the Gaussian curvature prescription problem on $\mathbb{S}^2$. Prior to the results of Gui-Moradifam \cite{GM-2018}, there had already been substantial progress in this direction, including the works of Feldman et al.~\cite{FFGG-1998}, Gui-Wei \cite{GW-2000}, C.~S.~Lin \cite{Lin1-2000}, and Ghoussoub-Lin \cite{GL-2010}. For a comprehensive exposition of these developments, we refer the reader to the book by Gui and Moradifam \cite{GM-2013}.

In the similar vein, the classical Nirenberg problem of prescribing Gaussian curvature on $\s^2$
 has generated a substantial body of research. For the reader’s convenience, we highlight only a few seminal contributions \cite{KW-1971, Hong-1986, CY-1987, CY-1988, Han-1990, Lin2-2000, Struwe-2005}, which laid the groundwork for much of the subsequent progress in this area.

\medskip

Let us now return to the Chang-Gui inequality \eqref{CGineq}. In addition to their striking inequality, Chang and Gui also studied the existence of extremizers for 
$I_{\alpha}$. Specifically, they considered the constrained minimization of 
$I_{\alpha}$ over the set with fixed center of mass:
 \begin{align*}
 \mathcal{M}_{a} : = \left\{ u \in H^1(\s^2) \ | \ \int_{\s^2} e^{2u} d \omega= 1, \int_{\s^2} \omega e^{2u} d\omega = a \right\}.
 \end{align*}
They showed that  $\inf_{u \in \mathcal{M}_{a}} I_{\alpha} (u) \geq 0$ if and only if $\alpha \geq \frac{2}{3}$. Moreover, a minimizer is attained and satisfies the Euler-Lagrange equation
 \begin{align}\label{el}
 \alpha \Delta_{\s^2} u(\omega) + \frac{1 - a \cdot \omega}{1 - |a|^2} e^{2u} - 1 = 0,
 \end{align}
where $\Delta_{\s^2}$ is the Laplace-Beltrami operator on $\s^2$ with respect to the standard metric. In addition, they showed that for $\alpha \neq \frac{2}{3},$ 
the only solution to \eqref{el} is the trivial solution $u \equiv 0$, whereas for $\alpha = \frac{2}{3}$, there exists a unique solution for each $a \in B(0,1)$ with center of mass $a = \int_{\s^2} \omega e^{2u} {\rm d}\omega$.
 They also provided an explicit expression for these solutions in stereographic coordinates when the center of mass lies along the $x_3$-axis.

\medskip

In this work, we investigate the quantitative stability of the inequality, seeking to understand how functions that nearly achieve equality must, in turn, remain close to the family of optimizers. Stability of geometric and functional inequalities is a classical theme whose history traces back to the works of Brezis-Lieb \cite{BL-1985} and the Bianchi-Egnell strategy \cite{BE-1991} in the context of Sobolev inequalities on $\R^n, n \geq 3$. The literature on this subject is vast, and it is beyond the scope of this article to provide a detailed account. By contrast, the stability theory for the Moser–Trudinger inequality is rather new, with only a few works available in this direction. To the best of our knowledge, the work of Chen et al.~\cite{CLT-2023} is the first of its kind, where the authors established a local stability result for the Trudinger–Moser–Onofri inequality on the sphere in terms of an $L^2$-distance involving $e^u$, following the approach pioneered by Bianchi and Egnell. However, in this case, due to inconsistencies between the homogeneity of the distance and that of the associated functional, one cannot expect a global stability result. In particular, the usual compactness argument that upgrades local to global stability fails in their setting.

\medskip

More recently, Carlen \cite{Car-2025} circumvented this complication by adopting a different strategy. He exploited the conformal invariance of the functional, as originally used by Onofri, together with the sharp Trudinger-Moser-Aubin inequality of Gui-Moradifam \cite{GM-2018}, to establish a stability result in terms of the gradient norm \ --- \ quite different from the distance considered in \cite{CLT-2023}. In this framework, the issue of homogeneity is resolved: Carlen observed that conformal invariance can be used to translate the center of mass of $e^{2u}$ for any $u \in H^1(\s^2)$ to zero, thereby enabling the application of Aubin’s inequality and yielding stability with respect to the gradient $L^2$-distance.

\medskip

Carlen’s approach serves as the starting point of our analysis. A closer examination of the solutions obtained by Chang and Gui uncovers a striking structural feature: when expressed in stereographic coordinates, their solution can be interpreted as essentially a $\frac{3}{2}$-multiple of the conformal factor arising in the classical Onofri inequality, modulo the addition of a suitable normalizing constant. This parallel naturally leads to the question of whether there exists an underlying conformal invariance associated with the functional $I_{\frac{2}{3}}$.

Motivated by this observation, we established the following invariance property. Let $u \in H^1(\s^2),$ and let $\tau : \s^2 \to \s^2,$ be  a conformal map. Define 
\begin{align*}
u_{\tau} = u \circ \tau + \frac{3}{4} \ln \mathcal{J}_{\tau} + c_{\tau},
\end{align*}
where $\mathcal{J}_{\tau}$ is the Jacobian of $\tau$ and $c_{\tau}$ is a normalizing constant chosen appropriately. With this definition, we verify in Section \ref{S conformal invariance}, Lemma \ref{conf-inv} that the functional is conformally invariant 
\begin{align*}
I_{\frac{2}{3}}(u_{\tau}) = I_{\frac{2}{3}}(u).
\end{align*}

This conformal invariance permits a precise characterization of the extremizers of $I_{\frac{2}{3}}(u_{\tau})$. Specifically, the complete set of extremizers can be described as
\begin{align*}
\mathcal{M}_{\mbox{\tiny{ext}}} =  \left\{ \psi_{\tau} :=  \frac{3}{4}\ln \mathcal{J}_{\tau} + c_{\tau} \ | \ \tau \ \mbox{is a conformal map of} \ \s^2\right\}
\end{align*}
up to the addition of constants.

Having identified the extremal family explicitly, the natural next step is to investigate the quantitative stability of the functional $I_{\frac{2}{3}}$. The central question is whether the deficit in $I_{\frac{2}{3}}$ provides quantitative control, from above, on the $H^1$-distance of a given function to the manifold of extremizers 
$\mathcal{M}_{\mbox{\tiny{ext}}}$. The conformal invariance established above plays a crucial role in this direction, as it enables us to adapt, within this framework, the method recently introduced by Carlen \cite{Car-2025}.

\medskip

We prove the following stability theorem:\

\begin{Th}\label{main}
For every $u \in H^1(\s^2),$
\begin{align*}
I_{\frac{2}{3}}(u) \geq \frac{1}{6} \inf_{\psi \in \mathcal{M}_{\mbox{\tiny{ext}}}} \int_{\s^2} |\nabla_{\s^2}(u - \psi)(\omega)|^2 \ d\omega
\end{align*}
holds.
\end{Th}

The article is organized as follows: in Section \ref{snotation} we recall elementary and well known results and terminologies. In Section \ref{sbasiclemmas}, we deduce some computations related to the Jacobian of a conformal map, some are classical, some are possibly new. We use this computation to derive several relations of the mass and center of mass of $\mathcal{J}_{\tau}^{\frac{3}{2}}.$ The results of Section \ref{sbasiclemmas} helps us to derive the conformal invariance of $I_{\frac{2}{3}}$ in Section \ref{S conformal invariance}. Finally, in Section \ref{sstability}, we prove the classification of the extremals of $I_{\frac{2}{3}},$ in terms of the conformal transformations of $\s^2$ and establish the stability result.

\section{Notations and Preliminaries} \label{snotation}
We will identify $\mathbb{R}^2$ with the complex plane $\mathbb{C}$ via 
\begin{align*}
z = x_1+ix_2, \ \ \ \ \ \ \ \ \ \ \ x = (x_1, x_2) \in \mathbb{R}^2,
\end{align*}
and denote by $dx$ the Lebesgue measure on $\mathbb{R}^2.$ 

Let $N = (0,0,1) \in \s^2$ be the north pole. The stereographic projection with respect to $N$ is the map 
\begin{align*}
\mathcal{S} : \s^2 \setminus \{N\} \to \mathbb{R}^2, \ \ \mathcal{S}(\omega)=\left(\frac{\omega_1}{1-\omega_3},\frac{\omega_2}{1-\omega_3}\right), \  \omega:=(\omega_1,\omega_2,\omega_3)\in \s^2 \setminus \{N\},
\end{align*}
with inverse transformation
\begin{align*}
\mathcal{S}^{-1}(x)=\frac{1}{1+|x|^2}\left(2x_1,2x_2,|x|^2-1\right), \ \ x=(x_1,x_2)\in \R^2.
\end{align*}
The Jacobian of $\mathcal{S}^{-1}$ as a map from $\R^2$ to $\R^3$ is $\left(\frac{2}{1+|x|^2}\right)^2$. If $d\sigma$ denotes the (un-normalized) surface area element, then by area formula we have 
\begin{align*}
\int_{\s^2}\varphi(\omega)\ds(\omega)=\int_{\R^2}\varphi\circ \mathcal{S}^{-1}(x)\left(\frac{2}{1+|x|^2}\right)^2{\rm d} x.
\end{align*}
In particular, $\sigma(\s^2)=4\pi.$ We denote by ${\rm d}\omega$ the normalised surface measure 
\begin{align*}
{\rm d}\omega=\frac{1}{4\pi}{\rm d}\sigma,
\end{align*}
so that 
\begin{align*}
\int_{\s^2}\varphi(\omega) {\rm d}\omega=\int_{\s^2}\varphi\circ \mathcal{S}^{-1}(x)\frac{1}{\pi (1+|x|^2)^2} {\rm d} x.
\end{align*}

\medskip

A conformal transformation of $\s^2$ is an angle preserving map $\tau : \s^2 \to \s^2.$ 
Via stereographic projection, such transformations can be identified with fractional linear (M\"{o}bius) transformations on $\mathbb{C}:$
\begin{align*}
z \mapsto \frac{az + b}{cz + d}, \ \ ad - bc = 1, \ \ a,b,c,d \in \mathbb{C}.
\end{align*}
Thus, the conformal group of $\s^2$ is a $6$-dimensional Lie group.

For $\lambda >0$ we denote by $\tau_{\lambda}$ the dilation by $\lambda$ and for $p \in \s^2,$ $\tau_p$ translation that takes the south pole to $p.$ In stereographic coordinates, these correspond to
\begin{align*}
\mathcal{S} \circ \tau_{\lambda} \circ \mathcal{S}^{-1}(x) = \lambda x, \ \ \ \ \mathcal{S} \circ \tau_{p} \circ \mathcal{S}^{-1}(x) =  x+ \mathcal{S}(p).
\end{align*}
Similarly, the inversion $x \mapsto \frac{x}{|x|^2},   x \in \mathbb{R}^2$ induces a conformal transformation of $\s^2$.

In summary, the conformal group of $\s^2$ (identified with the M\"{o}bius group of the Riemann sphere via stereographic projection) is generated by similarities of $\R^2$, that is, compositions of translations, rotations, and dilations together with the inversion $x \mapsto \frac{x}{|x|^2}.$

\medskip

Consider the Minkowski space-time $\R^{1,3} = \{(t, q) : t \in \R, q = (q_1,q_2,q_3) \in \R^3\}$ with  quadratic form 
\begin{align*}
\|(t,q)\|^2 = t^2 - |q|^2 = t^2 - (q_1^2 + q_2^2 + q_3^2),
\end{align*}
induced by the Lorentzian inner product $(t_1, q_1) \odot (t_2, q_2) = t_1t_2 - q_1 \cdot q_2,$ where $\cdot$ denotes the Euclidean dot product. The light cone is the set of all vectors $\{(t, q) \ | \|(t,q)\| = 0\}$, and the future light cone is the subset for which $t>0.$ 
The orthogonal Lorentz group  $O(1,3)$ consisting of all $4\times 4$ real matrices that preserves the quadratic form $\| \cdot \| $ and hence preserves $\odot$, is given by
\begin{align*}
O(1,3) := \left\{M \in M(4, \R) \ | \ M^{tr} \eta M = \eta \right\}, \ \ \ \eta = \mbox{diag}\{1,-1,-1,-1\}.
\end{align*}
Then $|\mbox{det} \ M| = 1,$ for $M \in O(1,3).$ Let $SO(1,3)$ be the special Lorentz group $\{M \in O(1,3) \ | \ \mbox{det} \ M = 1\}$ and denote by $SO^{+}(1,3)$ the proper orthochronous component.

\medskip

Identify $\R^{1,3}$ with the space $\mathcal{H}(2,\mathbb{C})$ of $2 \times 2$ Hermitian matrices via the linear bijection 
\begin{align*}
H : \R^{1,3} \mapsto \mathcal{H}(2,\mathbb{C}), \ \ \ \ \ \ H(t,q) = \begin{pmatrix}
    t + q_3  & q_1 + iq_2   \\
q_1 - iq_2      &   t - q_3
\end{pmatrix}.
\end{align*} 
Then $\mbox{det} \ H(t,q) = \|(t,q)\|^2.$ Thus real linear maps on $\R^{1,3}$ that preserve $\mbox{det} \ H$ are Lorentz transformations.

The special linear group 
\begin{align*}
SL(2,\mathbb{C}) = \left\{ \begin{pmatrix}
   a  & b   \\
c      &   d
\end{pmatrix}
\ | \ ad - bc = 1, a,b,c,d \in \mathbb{C}\right\}
\end{align*}
acts on $\mathcal{H}(2,\mathbb{C})$ via conjugation $H \mapsto AHA^{\star},$ $H \in \mathcal{H}(2,\mathbb{C}), A \in SL(2,\mathbb{C})$ which preserves determinant: $\mbox{det} \ AHA^{\star} = \mbox{det} \ H.$ Hence this action induces a Lorentz transformation on $\R^{1,3}$.

There exists a ``unique" homomorphism 
\begin{align*}
\Lambda : SL(2, \mathbb{C}) \mapsto SO(1,3)  
\end{align*}
such that
\begin{align} \label{Lorentz rep}
A H(t,q) A^{\star} = H (\Lambda(A)(t,q)), \ \ \ \mbox{for every} \ (t,q) \in \R^{1,3}.
\end{align}
Actually, the image of $\Lambda$ is the proper, orthochronous component $SO^{+}(1,3).$
The existence can be proved by verifying on the Pauli basis of $\mathcal{H}(2,\mathbb{C})$ or equivalently on the corresponding basis $\{(1,0,0,1), (1,0,0,-1), (0,1,1,0), (0,1, -1, 0)\}$ of $\R^{1,3}$. The homomorphism property can easily be verified through its action  
\begin{align*}
(AB)H(t,q)(AB)^{\star} = A [B H(t,q) B^{\star}]A^{\star} = A H(\Lambda(B)(t,q))A^{\star} = H(\Lambda(A)\Lambda(B)(t,q))
\end{align*}
and hence $\Lambda(AB) = \Lambda(A)\Lambda(B).$ If $A \in SL(2,\mathbb{C})$ acts trivially, then $AH(t,q)A^{\star} = H(t,q)$ for all $(t,q) \in \R^{1,3}$. This forces $A = \pm I$. Hence, $ker(\Lambda) = \{\pm I\}$ and $\Lambda$ is a $2$-$1$ covering map.

\section{Basic Technical Lemmas}\label{sbasiclemmas}

In this section, we derive several elementary, both classical and possibly new, technical computations related to the Jacobian of a conformal transformation on $\s^2.$
Let $\tau:\s^2\to \s^2$ be a conformal map. The Jacobian $\mathcal{J}_{\tau}(p)$ of $\tau$ at the point $p\in\s^2$ is given by the area distortion:
\begin{align*}
\mathcal{J}_{\tau}(p)=\lim _{r\to 0}\frac{\sigma(\tau(B(p,r)))}{\sigma(B(p,r))}, 
\end{align*}   
where $\sigma$ is the (un-normalized) surface area measure and
\begin{align*}
B(p,r)=\{q\in \s^2: d_{\s^2}(p,q)<r\},
\end{align*}
is the geodesic ball of radius $r$ centered at $p$. 
Geometrically, $B(p,r)$ is a spherical cap, and $d_{\s^2}$ denotes the geodesic distance on $\s^2$. 

\medskip

We next compute the area of a small geodesic ball. Note that 
\begin{align*}
\sigma(B(p,r))=2\pi(1-\cos r).
\end{align*}

Indeed, by rotational invariance we may assume $p = N$ is the north pole.
In spherical coordinates, the area element is $\sin\theta {\rm d}\theta {\rm d} \phi$, with $0\leq \theta\leq \pi, 0\leq \phi\leq 2\pi$.
The geodesic ball of radius $r$ centered at $N$ is $B(N,r)=\{(\sin \theta\sin \phi, \sin \theta \cos \phi,\cos\theta):0\leq \theta \leq r, 0\leq \phi\leq 2\pi\}$
so that
\begin{align*}
\sigma(B(N,r))=\int_0^{2\pi}\int_0^r \sin \theta {\rm d} \theta {\rm d}\phi = 2\pi(1-\cos r).
\end{align*}
Finally, by the Taylor expansion $\cos r=1-\frac{r^2}{2}+o(r^2)$, we obtain 
\begin{align*}
\sigma(B(p,r))=\pi r^2 +o(r^2) \ \ \ \mbox{as}\ \ \ r\to 0.
\end{align*}

\subsection{The Jacobian of a conformal map}
For a  conformal map $\tau: \s^2 \mapsto \s^2$, we denote by 
\begin{align*}
T:\R^2 \to \R^2, \ \ \ \ \ \ T(x)= \mathcal{S}\circ \tau\circ \mathcal{S}^{-1}(x)
\end{align*}
the induced map under stereographic projection. Using complex coordinates, we also regard $T$ as a map
 $T: \mathbb{C} \mapsto \mathbb{C}.$ 
 Let $\mathcal{J}_{T}(x)$ be the Jacobian of $T$ as a map from $\R^2$ to $\R^2.$ Recalling the standard relation, we have
 \begin{align*}
 \mathcal{J}_{T}(x) = |T^{\prime}(z)|^2, \ \ \ \  x = (x_1, x_2),\  z = x_1 + ix_2,
 \end{align*}
where $T^{\prime}$ is the complex derivative of $T.$
\begin{Lem}\label{Jacobian}
Let $\tau : \s^2 \to \s^2$ be a conformal map and $T$ be its stereographic representative.
Then, for every $\omega \in \s^2$, the Jacobian of $\tau$ at $\omega$ is given by 
\begin{align*}
\mathcal{J}_{\tau}(\omega) = \mathcal{J}_{T}(\mathcal{S}(\omega))\left(\frac{1+|\mathcal{S}(\omega)|^2}{1+|T(\mathcal{S}(\omega))|^2}\right)^2 = 
|T^{\prime}(\mathcal{S}(\omega))|^2\left(\frac{1+|\mathcal{S}(\omega)|^2}{1+|T(\mathcal{S}(\omega))|^2}\right)^2.
\end{align*}
\end{Lem}

\begin{proof} 
Fix $p \in \s^2$, by area formula,
\begin{align*}
\sigma(\tau(B(p,r))) &=\int_{\R^2} \chi_{\tau(B(p,r))}(\mathcal{S}^{-1}(x))\left(\frac{2}{1+|x|^2}\right)^2{\rm d} x\\
&=\int_{\R^2}\chi_{\mathcal{S}\circ \tau \circ \mathcal{S}^{-1}(\mathcal{S}(B(p,r)))}(x)\left(\frac{2}{1+|x|^2}\right)^2{\rm d} x\\
&=\int_{\R^2}\chi_{T(\mathcal{S}(B(p,r)))}(x)\left(\frac{2}{1+|x|^2}\right)^2{\rm d} x\\
&=\int_{T(\mathcal{S}(B(p,r)))}\left(\frac{2}{1+|x|^2}\right)^2{\rm d} x\\
&=\int_{\mathcal{S}(B(p,r))}\left(\frac{2}{1+|T(y)|^2}\right)^2\mathcal{J}_{T}(y) \ {\rm d} y\\
&=\int_{\mathcal{S}(B(p,r))}\mathcal{J}_{T}(y)\left(\frac{1+|y|^2}{1+|Ty|^2}\right)^2\left(\frac{2}{1+|y|^2}\right)^2{\rm d} y.
\end{align*}
In the second last line we used the change of variable formula.
We call $\rm d \mu=\left(\frac{2}{1+|y|^2}\right)^2 {\rm d} y$ that is a Radon measure on $\R^2$. 
Note that $\mu(\mathcal{S}(B(p,r))=\sigma(B(p,r))$ and since $\cap_{r>0}B(p,r)=\{p\}$, and $\mathcal{S}$ is 1-1 map,
$\cap_{r>0}\mathcal{S}(B(p,r))=\{\mathcal{S}(p)\}.$
As a result by Lebesgue-Besicovitch differentiation theorem, we deduce
\begin{align}\label{Jacobian Formula}
\mathcal{J}_\tau(p)=\lim_{r\to 0}\frac{\sigma(\tau(B(p,r))}{\sigma(B(p,r))}=\mathcal{J}_{T}(\mathcal{S}(p))\left(\frac{1+|\mathcal{S}(p)|^2}{1+|T(\mathcal{S}(p))|^2}\right)^2,
\end{align}
completes the proof.
\end{proof}

\begin{Rem}\label{Jacobian rem}
In complex stereographic coordinates if $\tau$ is given by the fractional linear transformation $\frac{az+b}{cz + d},$ then for every $\omega \in \s^2$
\begin{align*}
\mathcal{J}_{\tau}(\omega) = \left(\frac{1 + |z|^2}{|az + b|^2 + |cz + d|^2}\right)^2, \ \ \mbox{where}
\ z = \mathcal{S}(\omega).   
\end{align*}
 
\end{Rem}

A direct consequence of the formula \eqref{Jacobian Formula} is the following 

\begin{Cor}\label{jacobian cor}
Let $\tau: \s^2 \mapsto \s^2$ be a conformal map, $\omega \in \s^2$ and let $x = \mathcal{S}(\omega).$ 
\begin{itemize}
\item[(a)] If $\tau = \tau_{\lambda}$ is a dilation, then $\mathcal{J}_{\tau}(\omega)=\lambda^2\left( \frac{1+|x|^2}{1+\lambda^2|x|^2}\right)^2.$
\item[(b)] If $\tau = \tau_{p}$ is a translation, then $\mathcal{J}_{\tau}(\omega)=\left(\frac{1+|x|^2}{1+|x+\mathcal{S}(p)|^2}\right)^2.$
\item[(c)] If $\tau$ is either an orthogonal transformations or an inversion of $\R^2$ then $\mathcal{J}_{\tau}\equiv 1.$
\end{itemize} 
\end{Cor} 
\begin{proof}
If $\tau$ is an inversion of $\R^2$ then its Jacobian is $\frac{1}{|x|^4}$, and therefore the proof follows.
\end{proof}
 
 \subsection{Computation of mass $\int_{\s^2}\mathcal{J}_{\tau}^{\frac{3}{2}}{\rm d} \omega$}
 To establish the conformal invariance of the Chang–Gui inequality \eqref{CGineq}, we first compute the mass associated with a conformal map $\tau$. For brevity, set
\begin{align*}
M_{\tau}:=\int_{\s^2}\mathcal{J}_{\tau}^{\frac{3}{2}}{\rm d} \omega.
\end{align*}

\begin{Lem} \label{mass}
We have the following identities:
\begin{itemize}
\item[(a)] If $\tau = \tau_{\lambda}$ is a dilation, then  $M_{\tau}=\frac{1+\lambda^2}{2\lambda}$.
\item[(b)] If $\tau = \tau_p$ is a translation, then $M_{\tau}=\frac{1}{2}\left(\frac{3-p_3}{1-p_3}\right)$ where $p=(p_1,p_2,p_3)\in \s^2$.
\item[(c)] If $\tau$ is an orthogonal transformation of $\R^2$ or an inversion of $\R^2$ then $M_{\tau}=1$.
\end{itemize}
\end{Lem}

\begin{proof}
\begin{itemize}
\item[(a)] First we compute for $\tau_{\la}$. 
 \begin{align*}
  \int_{\s^2}\mathcal{J}_{\tau_{\la}}^{\frac{3}{2}}(\omega){\rm d} \omega&=\int_{\R^2}\mathcal{J}_{\tau_{\lambda}}(\mathcal{S}^{-1}(x))^{\frac{3}{2}}\frac{1}{\pi(1+|x|^2)^2} {\rm d} x \\
  &=\int_{\R^2}\left[\lambda^2\left(\frac{1+|x|^2}{1+\lambda^2|x|^2}\right)^2\right]^{\frac{3}{2}}\frac{1}{\pi(1+|x|^2)^2} {\rm d} x\\
  &=\frac{\lambda^3}{\pi}\int_{\R^2}\frac{1+|x|^2}{(1+\lambda^2 |x|^2)^3}{\rm d} x\\
  &=2\lambda^3 \left[\frac{1}{\la^2}\int_0^{\infty}\frac{s}{(1+s^2)^3}{\rm d}s +\frac{1}{\lambda^4}\int_0^{\infty}\frac{s^3}{(1+s^2)^3}{\rm d}x\right]\\
  &=\frac{1}{2}\left(\la+\frac{1}{\la}\right),
 \end{align*}
 where we used polar coordinates in the computation. 
 
  \item[(b)] Now consider  $\tau_{p}$. We set $\mathcal{S}(p)=\beta$.
\begin{align*}
\int_{\s^2}\mathcal{J}_{\tau_p}^{\frac{3}{2}}(\omega){\rm d}\omega &=\int_{\R^2}\mathcal{J}_{\tau_p}(\mathcal{S}^{-1}x)^\frac{3}{2}\frac{1}{\pi(1+|x|^2)^2} {\rm d} x\\
&=\int_{\R^2}\left(\frac{1+|x|^2}{1+|x+\beta|^2}\right)^3\frac{1}{\pi(1+|x|^2)^2} {\rm d} x\\
&=\frac{1}{\pi}\int_{\R^2}\frac{1+|y-\beta|^2}{(1+|y|^2)^3}{\rm d} y\\
&=\frac{1}{\pi}\int_{\R^2}\frac{1+|y|^2+|\beta|^2-2\beta\cdot y}{(1+|y|^2)^3}{\rm d} y\\
&=\frac{1}{\pi}\int_{\R^2}\frac{1}{(1+|y|^2)^2} dy+|\beta|^2\int_{\R^2}\frac{{\rm d}y}{(1+|y|^2)^3}\\
&=2\left[\int_0^{\infty}\frac{s}{(1+s^2)^2}+|\beta|^2\int_0^{\infty}\frac{s}{(1+s^2)^3}{\rm d} s\right]\\
&=1+\frac{1}{2}|\beta|^2,
\end{align*} 
 where again polar coordinates was used to reduce to one variable integration. Since
 \begin{align*}
 |\beta|^2=|\mathcal{S}(p)|^2=\frac{p_1^2+p_2^2}{(1-p_3)^2}=\frac{1+p_3}{1-p_3},
 \end{align*}
 the claimed identity follows.
 \item[(c)] 
This is immediate from $\mathcal{J}_{\tau}\equiv 1$. 
 \end{itemize} 
\end{proof}
\subsection{Computation of center of mass.} 

We define the center of mass (C.M.) 
\begin{align*}
a_{\tau}:=\frac{\int_{\s^2}\omega \mathcal{J}_{\tau}^{\frac{3}{2}}(\omega){\rm d}\omega}{\int_{\s^2}\mathcal{J}^{\frac{3}{2}}(\omega){\rm d}\omega}.
\end{align*}
Note that $|a_{\tau}| <1$ for every $\tau.$
\begin{Lem} \label{cm}
We have the following identities:
\begin{itemize}
\item[(a)] If $\tau = \tau_{\lambda}$ is a dilation, then  $a_{\tau}=(0,0,\frac{1-\lambda^2}{1+\lambda^2})$.
\item[(b)] If $\tau = \tau_p$ is a translation, then $a_{\tau}=\frac{1}{3-p_3}(-2p_1,-2p_2,1+p_3)$, where $p:=(p_1,p_2,p_3)\in \s^2$.
\item[(c)] If $\tau$ is either an orthogonal transformation or an inversion of $\R^2$ then $a_{\tau}=(0,0,0)$.
\end{itemize}
\end{Lem}
\begin{proof}
\begin{itemize}
\item[(a)] By Lemma \ref{mass}(a), $M_{\tau}=\frac{1+\la^2}{2\la}.$ We first evaluate
\begin{align*}
\int_{\s^2}\omega \mathcal{J}_{\tau}^{\frac{3}{2}}(\omega){\rm d}\omega &=\int_{\R^2}\mathcal{S}^{-1}(x) \mathcal{J}_{\tau}^{\frac{3}{2}}(\mathcal{S}^{-1}(x))\frac{1}{\pi(1+|x|^2)^2} dx\\
&=\frac{\la^3}{\pi}\int_{\R^2}\mathcal{S}^{-1}(x)\frac{1+|x|^2}{(1+\la^2|x|^2)^3},
\end{align*} 
and therefore the C.M. is given by
\begin{align*}
a_{\tau}=\frac{2\la^4}{\pi(1+\la^2)}\int_{\R^2}\mathcal{S}^{-1}(x)\frac{1+|x|^2}{(1+\la^2|x|^2)^3}{\rm d}x.
\end{align*}
Since $\mathcal{S}^{-1}(x)=\frac{1}{1+|x|^2}(2x_1,2x_2,|x|^2-1),$ the first two component of $a_{\tau}$ vanishes: $(a_{\tau})_1=(a_{\tau})_2=0.$ It remains to compute the last component.
\begin{align*}
(a_{\tau})_3&=\frac{2\la^4}{\pi(1+\la^2)}\int_{\R^2}\frac{|x|^2-1}{(1+\la^2|x|^2)^3}{\rm d}x\\
&=\frac{2\la^4}{\pi(1+\la^2)}\left[\frac{1}{\la^4}\int_{\R^2}\frac{|y|^2}{(1+|y|^2)^3}-\frac{1}{\la^2}\int_{\R^2}\frac{dy}{(1+|y|^2)^3}\right]\\
&=\frac{2\la^4}{(1+\la^2)}\left[\frac{1}{2\la^4}-\frac{1}{2\la^2}\right]\\
&=\frac{1-\la^2}{1+\la^2}.
\end{align*} 

\item[(b)] Now consider $\tau = \tau_{p}$, for some $p\in \s^2$. Denoting $\beta=\mathcal{S}(p) \in \mathbb{R}^2$,
\begin{align*}
\int_{\s^2}\omega \mathcal{J}_{\tau}^{\frac{3}{2}}(\omega){\rm d}\omega&=\frac{1}{\pi}\int_{\R^2} \mathcal{S}^{-1}(x)\frac{1+|x|^2}{(1+|x+\beta|^2)^3}{\rm d}x
\end{align*}
By Lemma \ref{mass}(b) $M_{\tau} =\frac{1}{2}\left(\frac{3-p_3}{1-p_3}\right)$ and therefore, the C.M. is given by
\begin{align*}
a_{\tau}=\frac{2(1-p_3)}{(3-p_3)\pi}\int_{\R^2}\frac{(2x_1,2x_2,|x|^2-1)}{(1+|x+\beta|^2)^3}{\rm d} x
\end{align*}
Now, we compute each component of $a_{\tau}$ separately. For $i=1,2$ 
\begin{align*}
\int_{\R^2} \frac{2x_i}{(1+|x+\beta|^2)^3}{\rm d}x
&=\int_{\R^2}\frac{2(y_i-\beta_i)}{(1+|y|^2)^3}{\rm d} y
=-2 \beta_i\int_{\R^2}\frac{{\rm d}y}{(1+|y|^2)^3}\\
&=-4\pi\beta_i\int_0^{\infty}\frac{s}{(1+s^2)^3}{\rm d}s
=-2\pi\beta_i\int_0^{\infty}\frac{\rm d t}{(1+t)^3}\\
&=-\frac{\pi p_i}{1-p_3}.
\end{align*}
Therefore, $(a_{\tau})_i=-\frac{2p_i}{(3-p_3)}$, for $i=1,2$. On the other hand, we have,
\begin{align*}
\int_{\R^2}\frac{|x|^2-1}{(1+|x+\beta|^2)^3}{\rm d} x
=&\int_{\R^2}\frac{|y-\beta|^2}{(1+|y|^2)^3}{\rm d} y-\int_{\R^2} \frac{{\rm d} y}{(1+|y|^2)^3}\\
=&\int_{\R^2}\frac{|y|^2}{(1+|y|^2)^3}{\rm d} y+(|\beta|^2-1)\int_{\R^2}\frac{{\rm d} y}{(1+|y|^2)^3}\\
=& \ \pi \int_0^{\infty}\frac{t{\rm d} t}{(1+t)^3}+\pi(|\beta|^2-1)\int_0^{\infty}\frac{{\rm d}t}{(1+t)^3}\\
=&\frac{\pi}{2}+\frac{\pi}{2}(|\beta|^2-1)=\frac{\pi}{2}|\beta|^2=\frac{\pi}{2}\frac{1+p_3}{1-p_3}.
\end{align*}
Therefore $(a_{\tau})_3=\frac{1+p_3}{3-p_3}.$ 
\item[(c)]  Follows from $\mathcal{J}_{\tau}\equiv 1$.
\end{itemize}
\end{proof}

As an immediate corollary, we arrive at the following identity, which will be particularly useful in our context.

\begin{Cor}\label{Nf1}
Let $\tau$ be one of the generators of the conformal maps: dilations, translations, orthogonal transformations and inversions of $\R^2$. Let $a$ be the C.M. of $\tau$ and define
\begin{align*}
c_{\tau}=-\frac{1}{2}\ln\int_{\s^2}\mathcal{J}_{\tau}^{\frac{3}{2}}{\rm d}\omega.
\end{align*}
Then the following relation holds
\begin{align*}
\e^{4c_{\tau}}=1-|a|^2. 
\end{align*}

\end{Cor}
\begin{proof}
The proof follows from direct verification. We first note that $\e^{4c_{\tau}}=M_{\tau}^{-2}$, where $M_{\tau}$ is the mass.

If $\tau$ is either an orthogonal transformation or an inversion, then $M_{\tau} = 1$ and $a_{\tau} = 0.$ Hence the relation holds trivially.

If $\tau = \tau_{\lambda}$ is a dilation, then  $M_{\tau}=\frac{1+\la^2}{2\la}$ and $|a_{\tau_{\la}}|^2=\left(\frac{1-\la^2}{1+\la^2}\right)^2$. From this we compute $1-|a_{\tau_{\la}}|^2=\frac{4\la^2}{(1+\la^2)^2}$. Thus the desired relation holds for dilations.

If $\tau = \tau_p$ is a translation, then $M_{\tau_{p}}=\frac{1}{2}\frac{3-p_3}{1-p_3}$ and $|a_{\tau_{p}}|^2=\left(\frac{4(p_1^2+p_2^2)+(1+p_3)^2}{(3-p_3)^2}\right).$ Hence 
\begin{align*}
1-|a_{\tau_p}|^2=\frac{(3-p_3)^2-[4(1-p_3^2)+(1+p_3)^2]}{(3-p_3)^2}=\frac{4(1+p_3^2-2p_3)}{(3-p_3)^2}=\frac{4(1-p_3)^2}{(3-p_3)^2}.
\end{align*}
Thus the identity also holds in this case.
\end{proof}

\subsection{The mass and C.M. relation of a conformal map}

\begin{Lem}\label{tauhalf}
Let $\tau : \s^2 \to \s^2$ be a conformal map of $\s^2$ and let $a \in B(0,1)$ be the C.M. of $\tau.$ Then the following relation holds
\begin{align}\label{Fundamental Requirement}
\mathcal{J}_{\tau}^{\frac{1}{2}}(\omega)=\frac{1-|a|^2}{1-a\cdot\omega}\int_{\s^2}\mathcal{J}_{\tau}^{\frac{3}{2}}(\xi){\rm d}\xi, \ \ \mbox{for every} \ \omega \in \s^2. 
\end{align} 
\end{Lem}

\begin{proof}
We will establish the relation \eqref{Fundamental Requirement} in the case when $\tau$ belongs to one of the generating transformations, namely dilations, translations, orthogonal transformations, and inversions of $\R^2$. Although the relation can in principle be verified for any conformal map by direct computation, we shall instead follow a different approach. Once the relation \eqref{Fundamental Requirement} is checked on the generators, we will prove in the next section the conformal invariance of 
$I_{\frac{2}{3}}$. As a consequence, the minimizers of $I_{\frac{2}{3}}$ are exactly those functions obtained from conformal transformations. This, in turn, implies that they satisfy the Euler–Lagrange equation \eqref{el}, from which we will finally deduce that \eqref{Fundamental Requirement} holds for every conformal map 
$\tau$.

If $\tau$ is an orthogonal transformation or an inversion of $\R^2$, then $\mathcal{J}_{\tau}\equiv 1$ and $a=0$, so \eqref{Fundamental Requirement} is obvious. So we assume first $\tau = \tau_{\lambda}$ is a dilation. Then by Lemma \ref{mass}(a) and Lemma \ref{cm}(a) we need to verify
\begin{align*}
\mathcal{J}_{\tau}^{\frac{1}{2}}(\omega) = \frac{1 - \left(\frac{1 - \lambda^2}{1 + \lambda^2}\right)^2}{1 - \frac{1 - \lambda^2}{1 + \lambda^2}w_3} \frac{1 + \lambda^2}{2\lambda} 
= \frac{2\lambda}{(1 + \lambda^2) - (1 - \lambda^2)w_3}.
\end{align*}
In stereographic coordinates, this becomes
\begin{align*}
\mathcal{J}_{\tau}^{\frac{1}{2}}(\mathcal{S}^{-1}(x)) &=  \frac{2\lambda}{(1 + \lambda^2) - (1 - \lambda^2)\frac{|x|^2 - 1}{|x|^2 + 1}}\\
&=\frac{2\lambda (|x|^2 + 1)}{(1 + \lambda^2)(|x|^2 + 1) - (1 - \lambda^2)(|x|^2 - 1)}\\
&=\frac{\lambda(1 + |x|^2)}{1 + \lambda^2|x|^2},
\end{align*}
which follows directly by Lemma \ref{jacobian cor}(a). Now suppose $\tau = \tau_p$ be a translation. Then by Lemma \ref{cm}(b), the C.M. is 
$a = \frac{1}{3-p_3}(-2p_1,-2p_2,1+p_3)$, where $p:=(p_1,p_2,p_3)\in \s^2$. Direct computation shows
\begin{align}\label{eq0}
1 - |a|^2 &= \frac{1}{(3 - p_3)^2} \left[(3 - p_3)^2 - 4(p_1^2 + p_2^2) - (1 + p_3)^2\right] \notag \\
&= \frac{1}{(3 - p_3)^2} (4 - 8p_3 + 4p_3^2) = \frac{4(1 - p_3)^2}{(3 - p_3)^2},
\end{align}
where we have used $p_1^2 + p_2^2 = 1 - p_3^2$ in the last line. On the other hand, 
\begin{align}\label{eq1}
1 - a \cdot \omega = \frac{1}{3 - p_3}\left[3 - p_3 + 2p_1\omega_1 + 2p_2\omega_2 - (1 + p_3)\omega_3\right]. 
\end{align}

Let $\beta = \mathcal{S}(p)$ and $x = \mathcal{S}(\omega).$ Then $3 - p_3 = 2 \ \frac{|\beta|^2 + 2}{|\beta|^2 + 1},$ and  a straightforward calculation yields
\begin{align*}
2p_1\omega_1 + 2p_2\omega_2 - (1 + p_3)\omega_3 =& \ 2 \left(\frac{2\beta_1}{1 + |\beta|^2}\frac{2x_1}{1 + |x|^2} + \frac{2\beta_2}{1 + |\beta|^2}\frac{2x_2}{1 + |x|^2}\right)  \\
& \ \ - \left(1 + \frac{|\beta|^2 - 1}{|\beta|^2 + 1}\right)\frac{|x|^2 - 1}{|x|^2 + 1} \\
&= \frac{1}{(1 + |\beta|^2)(1 + |x|^2)}[8 \beta \cdot x - 2 |\beta|^2(|x|^2-1)].
\end{align*}
Therefore, 
\begin{align}\label{eq2}
 & 3 - p_3 + 2p_1\omega_1 + 2p_2\omega_2 - (1 + p_3)\omega_3 \notag\\
 =& \frac{1}{(1 + |\beta|^2)(1 + |x|^2)} \left[2(|\beta|^2+2)(1 + |x|^2) + 8 \beta \cdot x - 2 |\beta|^2(|x|^2-1) \right] \notag \\
 =& \frac{4(1 + |x+\beta|^2)}{(1 + |\beta|^2)(1 + |x|^2)}.
\end{align}
Combining \eqref{eq1} and \eqref{eq2}, we deduce 
\begin{align}\label{f1}
1 - a \cdot \omega = \frac{2(1 + |x+\beta|^2)}{(2 + |\beta|^2)(1 + |x|^2)}.
\end{align}
By Lemma \ref{jacobian cor}(b), together with \eqref{eq0} and the identity $p_3 = \frac{|\beta|^2 - 1}{|\beta|^2 + 1}$, one also obtains
\begin{align}\label{f2}
(1 - |a|^2) \int_{\s^2}\mathcal{J}_{\tau}^{\frac{3}{2}}(\xi){\rm d}\xi = 2 \ \frac{1 - p_3}{3 - p_3}  = \frac{2}{|\beta|^2 + 2}.
\end{align}
Combining \eqref{f1} and \eqref{f2} gives
\begin{align*}
\frac{1-|a|^2}{1-a\cdot\omega}\int_{\s^2}\mathcal{J}_{\tau}^{\frac{3}{2}}(\xi){\rm d}\xi = \frac{1 + |x|^2}{1 + |x + \beta|^2} = \mathcal{J}_{\tau}^{\frac{1}{2}}(\omega).
\end{align*}
This completes the proof.
\end{proof}

\subsection{ Equation satisfied by Jacobian:}

Let $\tau:\s^2\to \s^2$ be a conformal map. We set 
\begin{align*}
\psi_{\tau}(\omega)=\frac{3}{4}\ln \mathcal{J}_{\tau}(\omega)+c_{\tau}, \ \ \omega \in \s^2, 
\end{align*}
where the constant $c_{\tau}$ (as in Corollary~\ref{Nf1}) is chosen so that 
\begin{align}\label{massnorm}
\int_{\s^2}\e^{2\psi_{\tau}}{\rm d}\omega =1.
\end{align} 
This gives,
\begin{align}\label{thectau}
\e^{-2c_{\tau}}=\int_{\s^2}\mathcal{J}_{\tau}^{\frac{3}{2}}{\rm d}\omega 
\ \ \ \mbox{or, equivalently,} \ \ \
c_{\tau}=-\frac{1}{2}\ln \int_{\s^2}\mathcal{J}_{\tau}^{\frac{3}{2}}{\rm d}\omega.
\end{align}
Let $a$ be the C.M. of $\tau,$ then  
\begin{align}\label{thecom}
\int_{\s^2}we^{2\psi_{\tau}(\omega)}{\rm d}\omega =  \frac{\int_{\s^2}\omega\mathcal{J}_{\tau}^{\frac{3}{2}}{\rm d}\omega}{\int_{\s^2}\mathcal{J}_{\tau}^{\frac{3}{2}}{\rm d}\omega} = a.
\end{align}

 We claim that  $\psi_{\tau}$ satisfies the Euler-Lagrange equation 
\begin{align}\label{Euler-Lagrange Eqn}
\frac{2}{3}\Delta_{\s^2}\varphi+\frac{1-a\cdot \omega}{1-|a|^2}\e^{2\varphi}-1=0\ \ \ \mbox{on}\ \ \ \s^2.
\end{align}
Indeed, since $v = \frac{1}{2}\ln\mathcal{J}_\tau$ solves $\Delta_{\s^2} v + e^{2v} - 1 = 0,$ we get
\begin{align*}
\frac{2}{3}\Delta_{\s^2}\psi_{\tau}(\omega)-1=\Delta_{\s^2}\left(\frac{1}{2}\ln\mathcal{J}_\tau\right)(\omega)-1
= \mathcal{J}_\tau(\omega)
\end{align*}
On the other hand, by Lemma \ref{tauhalf} and \eqref{thectau}
\begin{align*}
\frac{1-a\cdot\omega}{1-|a|^2}\e^{2\psi_{\tau}(\omega)} = \frac{1-a\cdot\omega}{1-|a|^2}e^{2c_{\tau}}\mathcal{J}_\tau(\omega)^{\frac{3}{2}} = \frac{1-a\cdot\omega}{1-|a|^2}\frac{\mathcal{J}_\tau(\omega)^{\frac{1}{2}}}{\int_{\s^2}\mathcal{J}_{\tau}^{\frac{3}{2}}(\omega){\rm d}\omega}\mathcal{J}_\tau(\omega) = \mathcal{J}_\tau(\omega).
\end{align*}

We remark once again that the proof of Lemma~\ref{tauhalf} was carried out only for the generators. Accordingly, the proof of the Euler–Lagrange equation \eqref{Euler-Lagrange Eqn} should also be understood as having been established only for the generators.


\begin{Cor}\label{zero}
We have  $I_{\frac{2}{3}}(\psi_{\tau})=0$ whenever $\tau$ is a dilation, a translation, an orthogonal transformation or an inversion of $\R^2$. 
\end{Cor}
\begin{proof} 
Since 
$\psi_{\tau}$ solves \eqref{Euler-Lagrange Eqn} subject to the constraints \eqref{massnorm} and \eqref{thecom}, the uniqueness theorem of Chang and Gui \cite[Proposition~2.2(ii)]{CG-2023} implies that  $\psi_{\tau}$ is the unique solution to \eqref{Euler-Lagrange Eqn}.

On the other hand, it is shown in \cite[Proposition~2.5 and the subsequent discussion]{CG-2023} that the constrained minimization problem
\begin{align*}
\inf I_{\frac{2}{3}}, \ \ \mbox{over the constraint \eqref{massnorm} and \eqref{thecom}},
\end{align*}
admits a minimizer, which solves \eqref{Euler-Lagrange Eqn}, and moreover $\inf I_{\frac{2}{3}} = 0$.
Therefore, we conclude $I_{\frac{2}{3}}(\psi_{\tau})=0.$
\end{proof}

We remark that, once the conformal invariance of $I_{\frac{2}{3}}$ is established in the next section, it will follow that
 $I_{\frac{2}{3}}(\psi_{\tau}) = 0$ for every $\tau.$

\section{The Conformal Invariance of $I_{\frac{2}{3}}$} \label{S conformal invariance}

We begin with the following lemma.

\begin{Lem}\label{fundamental}
For every conformal map $\tau : \s^2 \to \s^2,$ there exists a Lorentz transform $\Lambda_{\tau} \in SO^{+}(1,3)$ such that 
\begin{align*}
(1, \tau(\omega)) = \mathcal{J}_{\tau}(\omega)^{\frac{1}{2}} \Lambda_{\tau}(1, \omega),
\end{align*}
for every $\omega \in \s^2.$
\end{Lem}

\begin{proof}
For every $w \in \mathbb{S}^2,$ we express using the stereographic projection 
\begin{align*}
w = \frac{1}{1+ |z|^2} \left (2 Re \ z, 2 Im \ z, |z|^2 - 1\right).
\end{align*}
The vector $(1,w)$ belongs to the future light cone. We express
\begin{align*}
(1, w) = \frac{1}{1+ |z|^2} \left (1 + |z|^2, 2 Re \ z, 2 Im \ z, |z|^2 - 1\right).
\end{align*}
We apply the map $H : \R^{1,3} \mapsto \mathcal{H}(2, \mathbb{C}),$ and simplify
\begin{align*}
H(1,w) = \frac{2}{1 + |z|^2} \begin{pmatrix}
   |z|^2  & z   \\
\bar z    &   1
\end{pmatrix} 
 =  \frac{2}{1 + |z|^2}  \begin{pmatrix}
   z   \\
   1
   \end{pmatrix}
   \begin{pmatrix}
   \bar z & 1  
\end{pmatrix}.
\end{align*}
As a result, for $A \in SL(2,\mathbb{C})$
\begin{align}\label{conjugation}
AH(1,w)A^{\star} =  \frac{2}{1 + |z|^2}  \left[A\begin{pmatrix}
   z   \\
   1
   \end{pmatrix}\right]
  \left[A\begin{pmatrix}
   z   \\
   1
   \end{pmatrix}\right]^{\star}.
\end{align}
 
In complex stereographic coordinates, let $\tau$ be given by the fractional linear transformation $\frac{az + b}{cz + d}$ for some $A = \begin{pmatrix} a  & b   \\
 c    &   d \end{pmatrix} \in SL(2, \mathbb{C})$. Then we have
 \begin{align*}
 &\tau \circ \mathcal{S}^{-1}(z) = \mathcal{S}^{-1}\left(\frac{az + b}{cz + d}\right)\\
 = & \frac{1}{|az + b|^2 + |cz + d|^2} \left(2 Re (az + b)\overline{(cz + d)}, 2 Im (az + b)\overline{(cz + d)}, |az + b|^2 - |cz + d|^2\right).
 \end{align*}
Now consider $(1, \tau(\omega))$ in complex stereographic coordinates and apply the map $H$ to deduce
\begin{align}\label{htau}
H(1, \tau(\omega)) = \frac{2}{|az + b|^2 + |cz + d|^2} 
\begin{pmatrix} |az + b|^2  & (az + b)\overline{cz + d} \  \\ \\
 \ \overline{(az + b)}(cz + d)   &   |cz + d|^2
  \end{pmatrix}.
\end{align}
A direct calculation shows that 
\begin{align}\label{direct}
\left[A\begin{pmatrix}
   z   \\
   1
   \end{pmatrix}\right]
  \left[A\begin{pmatrix}
   z   \\
   1
   \end{pmatrix}\right]^{\star}
   &= \begin{pmatrix}
   az + b   \\
   cz + d
   \end{pmatrix} 
   \begin{pmatrix}
   \overline{az + b} &  \overline{cz + d} 
   \end{pmatrix} \notag \\
    & = \begin{pmatrix} |az + b|^2  & (az + b)\overline{cz + d} \  \\ \\
 \ \overline{(az + b)}(cz + d)   &   |cz + d|^2
  \end{pmatrix} \notag \\
  & = \frac{|az + b|^2 + |cz + d|^2}{2} \ H(1, \tau(\omega)).
\end{align}
By the discussions in Section 2, for this $A$, there exists a Lorentz transform $\Lambda_{\tau} \in SO^{+}(1,3)$ such that \eqref{Lorentz rep} holds.
Therefore, we conclude from \eqref{conjugation}, \eqref{htau}, \eqref{direct} and \eqref{Lorentz rep}
\begin{align*}
H(1, \tau(\omega))&= \frac{1 + |z|^2}{|az + b|^2 + |cz + d|^2} \ AH(1,\omega)A^{\star} \\
&= \frac{1 + |z|^2}{|az + b|^2 + |cz + d|^2} \ H(\Lambda_{\tau}(1,\omega)).
\end{align*}
Since $H$ is one-one, we conclude 
\begin{align*}
(1, \tau(\omega))= \frac{1 + |z|^2}{|az + b|^2 + |cz + d|^2} \ (\Lambda_{\tau}(1,\omega)).
\end{align*}
By Lemma \ref{Jacobian}, and the Remark \ref{Jacobian rem} after that, we conclude that the quantity 
$\frac{1 + |z|^2}{|az + b|^2 + |cz + d|^2}$ is precisely $\mathcal{J}_{\tau}(\omega)^{\frac{1}{2}},$ completing the proof of the lemma.
\end{proof}

The following result is the main statement of this section, establishing the conformal invariance of the functional 
$I_{\frac{2}{3}}.$ For reference, we recall that 
\begin{align*}
\psi_{\tau}(\omega)=\frac{3}{4}\ln \mathcal{J}_{\tau}(\omega)+c_{\tau}, \  c_{\tau} \ \mbox{denotes the normalizing constant}.
\end{align*}

\begin{Lem}\label{conf-inv}
Let $\tau : \s^2 \to \s^2$ be a conformal map and $u \in H^1(\s^2).$ We define
\begin{align*}
u_{\tau}:=u\circ \tau +\psi_{\tau}.
\end{align*}
Then 
\begin{align*}
  I_{\frac{2}{3}}(u_{\tau})=I_{\frac{2}{3}}(u)
 \end{align*}
holds.
 \end{Lem}
 \begin{proof}
We divide the proof into several steps. We first establish the result for the generators of the conformal group, and in the final step extend it to the entire conformal group.
 
 \medskip
 
{\bf Step 1:}
 We compute each of the terms separately.
\begin{align*}
\int_{\s^2}|\nabla_{\s^2} u_{\tau}|^2{\rm d}\omega&=\int_{\s^2}|\nabla_{\s^2}(u\circ \tau+\psi_{\tau})|^2{\rm d}\omega\\
&=\int_{\s^2}|\nabla_{\s^2}(u\circ\tau)|^2{\rm d}\omega+\int_{\s^2}|\nabla_{\s^2} \psi_{\tau}|^2{\rm d}\omega -2\int_{\s^2}(u\circ\tau) \Delta_{\s^2} \psi_{\tau}{\rm d}\omega.
\end{align*} 
Now using the equation \eqref{Euler-Lagrange Eqn} satisfied by $\psi_{\tau}$ we simplify
\begin{align*}
-2\int_{\s^2}(u\circ\tau) \Delta_{\s^2}\psi_{\tau}d\omega&=-2\cdot\frac{3}{2}\int_{\s^2}(u\circ\tau)\left[1-\frac{1-a\cdot\omega}{1-|a|^2}\e^{2\psi_\tau}\right]d\omega\\
&=-3\int_{\s^2}(u\circ\tau){\rm d}\omega +3\int_{\s^2}(u\circ\tau)\left(\frac{1-a\cdot\omega}{1-|a|^2}\right)\e^{2c_{\tau}}\mathcal{J}_{\tau}^{\frac{3}{2}}{\rm d}\omega\\
&=-3\int_{\s^2}(u\circ\tau){\rm d}\omega +3\int_{\s^2}(u\circ\tau) \mathcal{J}_{\tau}{\rm d}\omega\\
&=-3 \int_{\s^2}(u\circ\tau){\rm d}\omega+3\int_{\s^2}u{\rm d}\omega.
\end{align*}
Combining these identities, we get
\begin{align*}
\frac{2}{3}\int_{\s^2}|\nabla_{\s^2}u_{\tau}|^2\rm d\omega&=\frac{2}{3}\int_{\s^2}|\nabla_{\s^2} u|^2{\rm d}\omega +2\int_{\s^2}u\ {\rm d}\omega+\frac{2}{3}\int_{\s^2}|\nabla_{\s^2} \psi_{\tau}|^2{\rm d}\omega-2\int_{\s^2}(u\circ \tau) {\rm d}\omega.
\end{align*}
Putting it together with the identity $2\int_{\s^2}u_{\tau}{\rm d}\omega=2\int_{\s^2}(u\circ \tau){\rm d}\omega+2\int_{\s^2}\psi_{\tau}{\rm d}\omega$, we conclude
 \begin{align*}
 & \ \ \ \ \frac{2}{3}\int_{\s^2}|\nabla_{\s^2} u_{\tau}|^2{\rm d}\omega+2\int_{\s^2}u_{\tau}{\rm d}\omega \\
 &=\frac{2}{3}\int_{\s^2}|\nabla_{\s^2} u|^2{\rm d}\omega+2\int_{\s^2}u{\rm d}\omega+\int_{\s^2}|\nabla_{\s^2} \psi_{\tau}|^2{\rm d}\omega+2\int_{\s^2}\psi_{\tau}{\rm d}\omega\\
&=\frac{2}{3}\int_{\s^2}|\nabla_{\s^2} u|^2{\rm d}\omega+2\int_{\s^2}u{\rm d}\omega+\frac{1}{2}\ln(1-|a|^2)
 \end{align*}
 where in the last step we used the fact that $I_{\frac{2}{3}}(\psi_{\tau})=0$ for the generators, which follows from Corollary \ref{zero} together with the constraints \eqref{massnorm} and \eqref{thecom}.
Now we integrate the exponential terms,
\begin{align*}
\int_{\s^2}\e^{2u_{\tau}}{\rm d}\omega&=\int_{\s^2}\e^{2(u\circ\tau)+2\psi_{\tau}}{\rm d}\omega\\
&=\int_{\s^2}\e^{2(u\circ\tau)+\frac{3}{2}\ln \mathcal{J}_{\tau}+2c_{\tau}}{\rm d}\omega\\
&=\int_{\s^2}\e^{2(u\circ\tau)}\mathcal{J}_{\tau}\mathcal{J}_{\tau}^{\frac{1}{2}}\e^{2c_{\tau}}{\rm d}\omega\\
&=\int_{\s^2}\e^{2(u\circ\tau)} \frac{1-|a|^2}{1-a\cdot\omega}\mathcal{J}_{\tau}{\rm d} \omega\\
&=\int_{\s^2}\e^{2u} \frac{1-|a|^2}{1-a\cdot\tau^{-1}(\omega)}{\rm d} \omega.
\end{align*} 
Similarly,
\begin{align*}
\int_{\s^2}\omega\e^{2u_{\tau}}\rm d\omega&=\int_{\s^2}\omega \e^{2(u\circ\tau)}\mathcal{J}_{\tau}\frac{1-|a|^2}{1-a\cdot\omega}{\rm d}\omega\\
&=\int_{\s^2}\tau^{-1}(\omega) \e^{2u}\frac{1-|a|^2}{1-a\cdot\tau^{-1}(\omega)}{\rm d}\omega.
\end{align*}
Therefore, 
\begin{align*}
I_{\frac{2}{3}}(u_{\tau})&=\frac{2}{3}\int_{\s^2}|\nabla_{\s^2} u|^2{\rm d}\omega +2\int_{\s^2} u {\rm d}\omega+\frac{1}{2}\ln(1-|a|^2)\\
&\quad-\frac{1}{2}\ln\left[\left(\int_{\s^2}\frac{1-|a|^2}{1-a \cdot \tau^{-1}(\omega)}\e^{2u}{\rm d}\omega\right)^2-\left|\int_{\s^2}\tau^{-1}(\omega)\frac{1-|a|^2}{1-a \cdot \tau^{-1}(\omega)}\e^{2u}{\rm d}\omega\right|^2\right].
\end{align*}

{\bf Step 2:} Now we use the previous lemma to observe an invariance for the term within the logarithmic term.

Recall that by Lemma \ref{tauhalf} (known to be true for generators)
\begin{align*}
\mathcal{J}_{\tau}^{\frac{1}{2}}(\omega)&=\frac{1-|a|^2}{1-a\cdot\omega}\e^{-2c_{\tau}}
=\frac{(1-|a|^2)}{1-a\cdot\omega}\int_{\s^2}\mathcal{J}_{\tau}^{\frac{3}{2}}{\rm d}\omega.
\end{align*}
and by Lemma \ref{fundamental}, there exists a Lorentz transform $\Lambda_{\tau}$ such that
\begin{align*} 
(1,\tau(\omega))&=\mathcal{J}_{\tau}^{\frac{1}{2}}(\omega)\Lambda_{\tau}(1,\omega)=e^{-2c_{\tau}}\frac{(1-|a|^2)}{1-a\cdot \omega}\Lambda_{\tau}(1,\omega),
\end{align*}
or, equivalently,
\begin{align} \label{anidentity}
(1,\omega)=e^{-2c_{\tau}}\frac{(1-|a|^2)}{1-a\cdot\tau^{-1}(\omega)}\Lambda_{\tau}(1,\tau^{-1}(\omega)).
\end{align}
Multiplying \eqref{anidentity} by $e^{2u}$ and integrating  we get
\begin{align*}
\int_{\s^2}(1,\omega)\e^{2u}{\rm d}\omega =\e^{-2c_{\tau}}(1-|a|^2)\Lambda_{\tau}\left(\int_{\s^2}\frac{(1,\tau^{-1}(\omega))}{1-a\cdot\tau^{-1}(\omega)}\e^{2u}{\rm d}\omega\right).
\end{align*}
Since $\Lambda_{\tau}\in SO(3,1)$, it preserves the Lorentzian
\begin{align*}
 &\left(\int_{\s^2}\e^{2u}{\rm d}\omega\right)^2-\left|\int_{\s^2}\omega\e^{2u}{\rm d}\omega\right|^2
 \\
 &=\e^{-4c_{\tau}}(1-|a|^2)^2\left[\left(\int_{\s^2}\frac{1}{1-a \cdot \tau^{-1}(\omega)}\e^{2u}{\rm d}\omega\right)^2-\left|\int_{\s^2}\tau^{-1}(\omega)\frac{1}{1-a \cdot \tau^{-1}(\omega)}\e^{2u}{\rm d}\omega\right|^2\right]
\\
&=(1-|a|^2)\left[\left(\int_{\s^2}\frac{1}{1-a \cdot \tau^{-1}(\omega)}\e^{2u}{\rm d}\omega\right)^2-\left|\int_{\s^2}\tau^{-1}(\omega)\frac{1}{1-a \cdot \tau^{-1}(\omega)}\e^{2u}{\rm d}\omega\right|^2\right]. 
 \end{align*} 
In the last line, we used Corollary \ref{Nf1} which states that $(1-|a|^2)=e^{4c_{\tau}}$ for the generators.

\medskip

{\bf Step 3:}
By combining Step 1 and Step 2, we have established the conformal invariance for the generators:
\begin{align*}
I_{\frac{2}{3}}(u_{\tau})&=\frac{2}{3}\int_{\s^2}|\nabla_{\s^2} u|^2{\rm d}\omega +2\int_{\s^2}u{\rm d}\omega+\frac{1}{2}\ln(1-|a|^2)\\
&\quad-\frac{1}{2}\ln\left[\left(\int_{\s^2}\frac{1-|a|^2}{1-a\tau^{-1}(\omega)}\e^{2u}{\rm d}\omega\right)^2-\left|\int_{\s^2}\tau^{-1}(\omega)\frac{1-|a|^2}{1-a\tau^{-1}(\omega)}\e^{2u}{\rm d}\omega\right|^2\right]\\
&=\frac{2}{3}\int_{\s^2}|\nabla_{\s^2} u|^2{\rm d}\omega+2\int_{\s^2} u{\rm d}\omega -
\frac{1}{2}\ln\left[\left(\int_{\s^2}\e^{2u}{\rm d}\omega\right)^2-\left|\int_{\s^2}\omega\e^{2u}{\rm d}\omega\right|^2\right]\\
&=I_{\frac{2}{3}}(u).
\end{align*}
{\bf Step 4:}
If $\tau$ is a conformal map of $\s^2$, then by Liouville's theorem it can be expressed as a finite composition of the generating maps: dilations, translations, orthogonal transformations, and inversions. By chain rule we have $\psi_{\tau_1\circ\tau_2}=\psi_{\tau_1} \circ \tau_2+\psi_{\tau_2}+ c$ where c is a constant and since $I_{\frac{2}{3}}$ is invariant under the addition of constants
\begin{align*}
I_{\frac{2}{3}}(u_{\tau_1\circ\tau_2})&=I_{\frac{2}{3}}(u(\tau_1\circ\tau_2)+\psi_{\tau_1\circ\tau_2})\\
&=I_{\frac{2}{3}}((u \circ \tau_1)\circ\tau_2+\psi_{\tau_1}\circ \tau_2+\psi_{\tau_2}+ c)\\
&=I_{\frac{2}{3}}(u \circ \tau_1+\psi_{\tau_1})=I_{\frac{2}{3}}(u).
\end{align*}
Therefore, by iteratively applying the above observation to each factor in the decomposition of $\tau$, we obtain the desired conclusion.
 \end{proof}

\begin{Rem}
As a direct consequence of Lemma \ref{conf-inv}, by taking $u \equiv 0$, we obtain $I_{\frac{2}{3}}(\psi_{\tau}) = 0$ for every conformal map $\tau.$ Moreover, if $a$ denotes the center of mass of $e^{2\psi{\tau}},$ then $\psi_{\tau}$ is a solution of equation \eqref{Euler-Lagrange Eqn}. Consequently, Corollary \ref{tauhalf} holds for all conformal maps $\tau$.
\end{Rem}

\section{Classification of Extremal Set and Stability}\label{sstability}

In this section, we establish a complete classification of the extremizers of $I_{\frac{2}{3}}$ in terms of the conformal group of 
$\s^2$. It is immediate that the inequality admits a minimizer, since substituting $u \equiv 0$ yields equality. However, Chang and Gui \cite{CG-2023}, in the course of proving the inequality, demonstrated that for any admissible choice of center of mass, one can construct a nontrivial extremizer. This naturally raises the question of whether the family of extremizers forms a three-dimensional manifold, modulo the trivial freedom of adding constants, which merely rescales the mass $\int_{\s^2} e^{2u} {\rm d}\omega$.

In what follows, building on the work of Chang and Gui, we provide a complete classification of the extremizers, showing that the set of extremizers admits a natural three-dimensional parametrization (modulo additive constants), arising from the action of the conformal group on $\mathbb{S}^2$. This classification clarifies the precise geometric nature of the extremizers and highlights the central role of conformal invariance in the problem.

\begin{Th}\label{Classification}
Let $u \in H^1(\s^2).$ Then 
\begin{align*}
I_{\frac{2}{3}}(u)=0 \ \ \ \ \ \  \mbox{if and only if} \  \ \ \ \ \ u=\frac{3}{4}\ln \mathcal{J}_{\tau}+ c,
\end{align*}
where $\tau$ is a conformal transformation of $\mathbb{S}^2$, $\mathcal{J}_{\tau}$ denotes its Jacobian, and $c \in \mathbb{R}$ is an arbitrary constant.
\end{Th} 

As a direct consequence of Theorem \ref{Classification}, the set of extremizers, up to an addition of a constant, is given by
\begin{align*}
\mathcal{M}_{\mbox{\tiny{ext}}}=\left\{\frac{3}{4}\ln \mathcal{J}_{\tau}+ c_{\tau} \ | \ \tau : \s^2 \to \s^2 \ \mbox{conformal map} \right\}.
\end{align*}

It is worthwhile emphasizing that the subgroup of conformal maps arising from orthogonal transformations $O(3)$ of $\R^3$
 leaves the functional $I_{\frac{2}{3}}$ invariant. Since the orthogonal group $O(3)$ has dimension three, these symmetries merely reproduce the trivial extremizers (constant solutions). In contrast, the non-trivial extremizers emerge precisely from the additional conformal symmetries of $\s^2$, captured by the full action of $PSL(2,\mathbb{C}).$

The proof of the theorem follows the same line of reasoning as Onofri \cite{O-1982}. To carry out the argument we require a preparatory lemma which asserts that, for any admissible function, there exists a conformal change that moves its center of mass to the origin (see also the variational proof in \cite{Car-2025}).

Recall the notation 
\begin{align*}
\psi_{\tau}:=\frac{3}{4}\ln\mathcal{J}_{\tau}+c_{\tau}, \ \ \ \ \mbox{and} \ \ \ \ u_{\tau} = u \circ \tau + \psi_{\tau},
\end{align*}
where $c_{\tau}$ is the normalizing constant chosen so that the mass constraint $\int_{\s^2}e^{2\psi_{\tau}} {\rm d}\omega = 1$ (cf. \eqref{massnorm}) holds.

\begin{Lem}\label{Zero C.O.M.}
For $u\in H^1(\s^2)$, there exist a conformal map $\tau: \s^2 \to \s^2$ such that 
\begin{align*}
\int_{\s^2}\omega\e^{2 u_{\tau}}d\omega=0.
\end{align*}
\end{Lem}
\begin{proof}
Via stereographic projection $\mathcal{S}$, we reduce the problem to $\R^2$. We consider conformal maps given by a translation by $x_0 \in \R^2$ followed by a dilation $\lambda_0 >0$. Together, these transformations provide the three degrees of freedom needed to adjust the center of mass to the origin.

For brevity, we denote this specific conformal map by
 $\tau(\omega):=\mathcal{S}^{-1}(\lambda \mathcal{S}(\omega)+x_0)$ and the associated extremizer by $\psi_{\tau}$. In stereographic coordinates, this takes the explicit form
\begin{align*}
\psi_{\tau}(\mathcal{S}^{-1}(x)):=\frac{3}{4}\ln\frac{\lambda_0^2(1+|x|^2)^2}{(1+|\lambda_0x+x_0|^2)^2}+c_{\tau}.
\end{align*}

We then consider the transformed function $u\circ\tau+\psi_{\tau}.$ For $i=1,2$, we expand and use change of variable,
\begin{align*}
\int_{\s^2}\omega_i \e ^{2 u_{\tau}}{\rm d} x &=2\e^{2 c_{\tau}}\int_{\R^2} x_i \e^{2u(\mathcal{S}^{-1}(\lambda_0 x + x_0))}\frac{\lambda_0^3}{(1+|\lambda_0 x+x_0|^2)^3}{\rm d}x\\
&=2\e^{2 c_{\tau}}\int_{\R^2}(x-x_0)_i \e^{2u\circ \mathcal{S}^{-1}(x)}\frac{\rm dx}{(1+|x|^2)^3}.
\end{align*}
Equating it to zero yields the constraint,
\begin{align*}
x_0={\left(\int_{\R^2}\frac{\e^{2u\circ \mathcal{S}^{-1}(x)}}{(1+|x|^2)^3}{\rm d}x\right)^{-1}}{\int_{\R^2}\frac{x}{(1+|x|^2)^3}\e^{2u\circ \mathcal{S}^{-1}(x)}{\rm d} x }.
\end{align*}
On the other hand,
\begin{align*}
\int_{\s^2}\omega_3 e^{2u_{\tau}}{\rm d} \omega &= 2\e^{2 c_{\tau}}\int_{\R^2} (|x|^2 -1)\e^{2u(\mathcal{S}^{-1}(\lambda_0 x + x_0))}\frac{\lambda_0^3}{(1+|\lambda_0 x+x_0|^2)^3}{\rm d}x\\
&=2\e^{2 c_{\tau}}\lambda_0\int_{\R^2}\left(\left|\frac{x-x_0}{\lambda_0}\right|^2-1\right) \e^{2u\circ \mathcal{S}^{-1}(x)}\frac{\rm dx}{(1+|x|^2)^3}.
\end{align*} 
The integral on the right diverges to $+\infty$ as $\lambda_0 \rightarrow 0+$, and to $-\infty$ as $\lambda_0 \rightarrow +\infty.$
By continuity, it follows that there exists a value of $\la_0$ for which the desired conclusion holds.
\end{proof}

\begin{Rem}
The proof shows that for any $u \in H^1(\s^2)$ there exists a unique conformal map (modulo orthogonal transformations) such that the center of mass of $e^{2u_{\tau}}$ vanishes. The corresponding parameter $\lambda_0$ is given by
\begin{align*}
\sqrt{\frac{\int_{\R^2}\frac{\e^{2u \circ \mathcal{S}^{-1}}}{(1+|x|^2)^2}{\rm d}x-(1+|x_0|^2)\int_{\R^2}\frac{e^{2u\circ \mathcal{S}^{-1}(x)}}{(1+|x|^2)^3}}{\int_{\R^2}\frac{\e^{2u\circ \mathcal{S}^{-1}(x)}}{(1+|x|^2)^3}{\rm d} x}}.
\end{align*}
Although positivity of the numerator is not immediately evident, but substituting the expression for $x_0$ from Lemma~\ref{Zero C.O.M.} and applying Jensen’s inequality will show that it is indeed positive.  
\end{Rem}

\noindent
{\bf Proof of the Theorem \ref{Classification}.}
\begin{proof}
We know that $I_{\frac{2}{3}}(\psi_{\tau})=0.$ Since $I_{\frac{2}{3}}(u+c)=I_{\frac{2}{3}}(u)$ for every $u\in H^1(\s^2)$ and $c\in\R$ it follows that 
\begin{align*}
I_{\frac{2}{3}}\left(\frac{3}{4}\mathcal{J}_{\tau}+c\right)=0 \ \mbox{for every conformal map} \ \tau \ \mbox{and} \ c\in \R.
\end{align*}

Conversely, let  $u\in H^1(\s^2)$  be such that $I_{\frac{2}{3}}(u)=0$. Lemma \ref{Zero C.O.M.} ensures the existence of a conformal map $\tau$ such that center of mass of $e^{2u_{\tau}}$ vanishes.
The conformal invariance of $I_{\frac{2}{3}}$ (c.f. Lemma \ref{conf-inv}) then yields
\begin{align*}
J_{\frac{2}{3}}(u_{\tau}) = I_{\frac{2}{3}} (u_{\tau})  = I_{\frac{2}{3}} (u) = 0.
\end{align*}
where $J_{\frac{2}{3}}$ is the functional associated with the Trudinger-Moser-Aubin inequality (c.f. \eqref{TMAubin}).
From the classification of extremizers due to Gui-Moradifam~\cite{GM-2018}, it follows that $u_{\tau} = c,$ for some $c \in \R.$ Finally the identity $\psi_{\tau} \circ \tau^{-1}=-\psi_{\tau^{-1}} + \ \mbox{constant}$ gives $u = \psi_{\tau^{-1}} + c$, which completes the proof.  
\end{proof}

We now present the proof of our stability result, adapting an argument of Carlen~\cite{Car-2025}.

\medskip

\noindent
{\bf Proof of Theorem \ref{main}.}
 
 \begin{proof}
Let $u \in H^1(\s^2).$ By Lemma \ref{Zero C.O.M.}, there exists a conformal map $\tau$ such that 
 \begin{align*} 
  \int_{\s^2}\omega \e^{2u_{\tau}}{\rm d}\omega=0.
 \end{align*}
By the result of Gui-Moradifam~\cite{GM-2018}, we have
 \begin{align*}
 J_{\frac{1}{2}}(u_{\tau})\geq 0,
 \end{align*}
and therefore,
\begin{align*}
J_{\frac{2}{3}}(u_{\tau})\geq \left(\frac{2}{3}-\frac{1}{2}\right) \int_{\s^2} |\nabla_{\s^2} u_{\tau}|^2 {\rm d}\omega.
\end{align*}
Since for any $\alpha$ the functionals satisfy $I_{\alpha}\geq J_{\alpha}$, it follows that 
 \begin{align*}
 I_{\frac{2}{3}}(u_{\tau})\geq \frac{1}{6}\int_{\s^2} |\nabla_{\s^2} u_{\tau}|^2 {\rm d}\omega.
 \end{align*}
 Using the conformal invariance of $I_{\frac{2}{3}}$ (Lemma \ref{conf-inv}) and the relation $\psi_{\tau^{-1}}= - \psi_{\tau}\circ \tau ^{-1} + \ \mbox{constant}$, we deduce  
\begin{align*} 
 I_{\frac{2}{3}}(u)\geq \frac{1}{6}\int_{\s^2} |\nabla_{\s^2} (u - \psi_{\tau^{-1}})|^2 {\rm d}\omega.
\end{align*}
Taking the infimum over all conformal maps $\tau$ completes the proof.
 \end{proof}

\medskip

\noindent
{\bf Acknowledgement.}  D. Karmakar acknowledges the support of the Department of Atomic Energy, Government of India,
under project no. 12-R\&D-TFR-5.01-0520.

\bibliographystyle{alpha}
\bibliography{Onofri.bib}
\end{document}